\documentclass[12pt]{amsart}
\usepackage{amsmath,amsthm,latexsym,amscd,amsbsy,amssymb,amsfonts,fleqn,leqno,
euscript, pb-diagram}

\numberwithin{equation}{section}

\newcommand{\I}{\mathbb I}

\newcommand{\reg}{\mathrm{e}}

\newtheorem{thm}{Theorem}[section]
\newtheorem{pro}[thm]{Proposition}
\newtheorem{lem}[thm]{Lemma}
\newtheorem{cor}[thm]{Corollary}
\newcommand{\Tee }{\mathcal T}

\begin{document}


\title[I-FAVORABLE SPACES: REVISITED]
{I-FAVORABLE SPACES: REVISITED}

\author{Vesko  Valov}
\address{Department of Computer Science and Mathematics, Nipissing University,
100 College Drive, P.O. Box 5002, North Bay, ON, P1B 8L7, Canada}
\email{veskov@nipissingu.ca}
\thanks{Research supported in part by NSERC Grant 261914-13}




\begin{abstract}
The aim of this paper is to extend the external characterization of $\mathrm I$-favorable spaces obtained in \cite{vv}.
This allows us to obtain a characterization of compact $\mathrm I$-favorable spaces in terms of quasi $\kappa$-metrics.
We also provide proofs of some author's results announced in \cite{v}.
\end{abstract}

\maketitle

\markboth{}{I-favorable spaces}



\section{Introduction}

The aim of this paper is to extend the external characterization of $\mathrm I$-favorable spaces obtained in \cite{vv}.
We also provide proofs of some author's results announced in \cite{v}.
All topological spaces are Tychonoff
and the single-valued maps are continuous. 

P. Daniels, K. Kunen and H. Zhou \cite{dkz} introduced the so called
{\em open-open game}: Two players take countably many turns, a round consists of player I choosing a non-empty open set $U\subset X$ and II choosing
a non-empty open set $V\subset U$. Player I wins  if the union of II's open sets is dense in $X$, otherwise II wins.
A space $X$ is called {\em
$\mathrm{I}$-favorable} if player I has
a winning strategy. This means, see \cite{kp1}, there exists a
function $\sigma:\bigcup_{n\geq 0}\Tee_X^n\to\Tee_X$ such that the
union $\bigcup_{n\geq 0}U_n$ is dense in $X$
for each game
$$\big(\sigma(\varnothing),U_0,\sigma(U_0),U_1,\sigma(U_0,U_1),U_2,...,U_n,\sigma(U_0,U_1,..,U_n),U_{n+1},,,\big),$$
where all $U_k$ and $\sigma(\varnothing)$ are non-empty open sets in $X$, $U_0\subset\sigma(\varnothing)$ and
$U_{k+1}\subset\sigma(U_0,U_1,..,U_k)$ for every $k\geq 0$ (here $\Tee_X$ is the topology of $X$).

Recently A. Kucharski and S. Plewik (see \cite{kp1}, \cite{kp2})
 investigated the connection of $\mathrm{I}$-favorable
spaces and skeletal maps. In particular, they proved in \cite{kp2}
that the class of compact $\mathrm{I}$-favorable spaces and the
skeletal maps are adequate in the sense of E. Shchepin \cite{sc1}. Recall that a map $f:X\to Y$ is skeletal if $\mathrm{Int}\overline{f(U)}\neq\varnothing$) for every open $U\subset X$.
On the other hand, the author announced \cite[Theorem 3.1]{v} a
characterization of the spaces $X$ such that there is an inverse
system $\displaystyle S=\{X_\alpha, p^{\beta}_\alpha, A\}$ of
separable metric spaces $X_\alpha$ and skeletal surjective bounding maps $p^{\beta}_\alpha$
satisfying the following conditions:
$(1)$ the index set $A$ is $\sigma$-complete (every countable chain in $A$ has a supremum in $A$);
$(2)$ for every countable chain $\{\alpha_n\}_{n\geq 1}\subset A$ with $\beta=\sup\{\alpha_n\}_{n\geq 1}$ the space $X_\beta$ is a (dense)
subset of $\displaystyle\lim_\leftarrow\{X_{\alpha_n},p^{\alpha_{n+1}}_{\alpha_n}\}$;
$(3)$ $X$ is embedded in $\displaystyle\lim_\leftarrow
S$ and $p_\alpha(X)=X_\alpha$ for each $\alpha$, where $p_\alpha\colon\displaystyle\lim_\leftarrow S\to X_\alpha$ is the $\alpha$-th limit projection.
An inverse system satisfying $(1)$ and $(2)$ is called {\em almost $\sigma$-continuous}. If condition $(3)$ is satisfied, we say that $X$ is the {\em almost limit of $S$}, notation $X=\mathrm{a}-\displaystyle\lim_\leftarrow S$. Spaces $X$ such that  $X=\mathrm{a}-\displaystyle\lim_\leftarrow S$, where $S$ is almost $\sigma$-continuous inverse system with skeletal bounding maps and second countable spaces, are called {\em skeletally generated} \cite{vv}.

The following theorem is our first main result:
\begin{thm}
For a space $X$ the following conditions are equivalent:
\begin{itemize}
\item[(1)] $X$ is $\mathrm{I}$-favorable;
\item[(2)] Every embedding of $X$ in another space $Y$ is $\pi$-regular;
\item[(3)] $X$ is skeletally generated.
\end{itemize}
\end{thm}

Here, we say that a subspace $X\subset Y$ is {\em $\pi$-regularly
embedded} in $Y$ \cite{v} if there exists a function
 $\reg\colon\Tee_X\to\Tee_Y$ such that for every $U,V\in\Tee_X$ we have:
 (i) $\reg(U)\cap\reg(V)=\varnothing$ provided $U\cap
V=\varnothing$; (ii) $\reg(U)\cap X$ is a dense subset of $U$.
 If, $\reg(U)\cap X=U$, we say that $X$ is {\em regularly embedded in $Y$}. An external characterization of $\kappa$-metrizable
 compacta, similar to condition (2), was established in \cite{s}.
\begin{cor}
Every $\mathrm I$-favorable subset of an extremally disconnected space
is also extremally disconnected.
\end{cor}

\begin{cor}
Every open subset of an $\mathrm I$-favorable space is $\mathrm I$-favorable.
\end{cor}

A version of Theorem 1.1 was established in \cite{vv}, but we used a little bit different notions. First, we considered
$\mathrm{I}$-favorable spaces with respect to the family of co-zero sets. Also, in the definition of skeletally generated spaces we required the system $S$ to be factorizable
(i.e. for each continuous function $f$ on $X$ there exists $\alpha\in A$ and a continuous function $h$ on $X_\alpha$ with $f=h\circ p_\alpha$). Moreover, in item (2) $X$ was supposed to be $C^*$-embedded in $Y$. Corollary 1.2 was also established in \cite{vv} under the assumption of $C^*$-embedability.

 Recall that a {\em $\kappa$-metric} \cite{sc1} on a space $X$  is a non-negative function $\rho(x,C)$ of two variables, a point $x\in X$ and a canonically closed set $C\subset X$, satisfying the following axioms:
\begin{itemize}
\item[K1)] $\rho(x,C)=0$ iff $x\in C$;
\item[K2)] If $C\subset C'$, then $\rho(x,C')\leq\rho(x,C)$ for every $x\in X$;
\item[K3)] $\rho(x,C)$ is continuous function of $x$ for every $C$;
\item[K4)] $\rho(x,\overline{\bigcup C_\alpha})=\inf_{\alpha}\rho(x,C_\alpha)$ for every increasing transfinite family $\{C_\alpha\}$ of canonically closed sets in $X$.
\end{itemize}
We say that a function $\rho(x,C)$ is an {\em quasi $\kappa$-metric} on $X$ if it satisfies the axioms $K2) - K4)$ and the following one:
\begin{itemize}
\item[K$1^*)$] For any $C$ there is a dense open subset $V$ of $X\setminus C$ such that $\rho(x,C)=0$ iff $x\in X\setminus V$.
\end{itemize}
Our second result provides a characterization of compact $\mathrm I$-favorable spaces, which is similar to Shchepin's characterization (\cite{sc1}, \cite{sc2}) of openly generated compacta as compact spaces admitting a $\kappa$-metric.
\begin{thm}
A  compact space $X$ is $\mathrm I$-favorable iff $X$ is quasi $\kappa$-metrizable.
\end{thm}

\begin{cor}
Every $\mathrm I$-favorable space is quasi $\kappa$-metrizable.
\end{cor}

The paper is organized as follows:
Section 2 contains the proof of Theorem 1.1 and Corollaries 1.2-1.3. The proofs of Theorem 1.4 and Corollary 1.5 are contained in section 3. In section 4 we provide the proof of some results concerning almost continuous inverse systems with nearly open bounding maps, which were announced in \cite{v}.

\section{Proof of Theorem 1.1}

If follows from the definition of $\mathrm{I}$-favorability that a given space is $\mathrm{I}$-favorable if and only if there are
a $\pi$-base $\mathcal B$ and a function $\sigma:\bigcup_{n\geq 0}\mathcal B^n\to\mathcal B$ such that the
union $\bigcup_{n\geq 0}U_n$ is dense in $X$ for any sequence
$$\big(\sigma(\varnothing),U_0,\sigma(U_0),U_1,\sigma(U_0,U_1),U_2,...,U_n,\sigma(U_0,U_1,..,U_n),U_{n+1},,,\big),$$
where $U_k$ and $\sigma(\varnothing)$ belong to $\mathcal B$, $U_0\subset\sigma(\varnothing)$ and
$U_{k+1}\subset\sigma(U_0,U_1,..,U_k)$ for every $k\geq 0$. Such a function will be also called a winning strategy. Recall that $\mathcal B$ is a
$\pi$-base for $X$ if every open set in $X$ contains an element from $\mathcal B$.
\begin{pro}\cite{ku}
Let $\mathcal B$ and $\mathcal P$ be two $\pi$-bases for $X$. Then there is a winning strategy $\sigma:\bigcup_{n\geq 0}\mathcal B^n\to\mathcal B$
if and only if there is a winning strategy $\mu:\bigcup_{n\geq 0}\mathcal P^n\to\mathcal P$.
\end{pro}

\begin{proof}
Suppose $\sigma:\bigcup_{n\geq 0}\mathcal B^n\to\mathcal B$ is a winning strategy. We define a winning strategy $\mu:\bigcup_{n\geq 0}\mathcal P^n\to\mathcal P$ by induction. We choose any open non-empty set $\mu(\varnothing)\in\mathcal P$ such that $\mu(\varnothing)\subset\sigma(\varnothing)$. If
$V_0\in\mathcal P$ is the answer of player II in the game played on $\mathcal P$ (i.e., $V_0\subset\mu(\varnothing)$), then we choose $U_0\in\mathcal B$ with $U_0\subset V_0$ ($U_0$ can be considered as the answer of player II in the game played on $\mathcal B$). Assume we already defined $V_0,..,V_{n}\in\mathcal P$ and
$U_0,..,U_{n}\in\mathcal B$ such that $U_{k+1}\subset V_{k+1}\subset\mu(V_0,..,V_k)\subset\sigma(U_0,..,U_k)$ for all $k\leq n-1$. Then, we choose
$\mu(V_0,..,V_n)\in\mathcal P$ such that $\mu(V_0,..,V_n)\subset\sigma(U_0,..,U_n)$. If $V_{n+1}\in\mathcal P$ is the choice of player II in the game played
on $\mathcal P$ such that $V_{n+1}\subset\mu(V_0,..,V_n)$, we choose $U_{n+1}\in\mathcal B$ with $U_{n+1}\subset V_{n+1}$. This complete the induction.
Since $\sigma$ is a winning strategy and $U_{k}\subset V_{k}$ for each $k$, the union $\bigcup_{n\geq 0}V_n$ is dense in $X$. So, $\mu$ is also a winning strategy.
\end{proof}
In \cite{vv} we considered $\mathrm{I}$-favorable spaces $X$ with respect to the co-zero sets meaning that there is a winning strategy
$\sigma:\bigcup_{n\geq 0}\Sigma^n\to\Sigma$, where $\Sigma$ is the family of all co-zero subsets of $X$. Proposition 2.1 shows that this is equivalent to $X$ being $\mathrm{I}$-favorable. So, all results from \cite{vv} are valid for $\mathrm{I}$-favorable spaces.

According to \cite[Corollary 1.4]{dkz}, if $Y$ is a dense subset of $X$, then $X$ is $\mathrm{I}$-favorable if and only $Y$ is $\mathrm{I}$-favorable.
So, every compactification of a space $X$ is $\mathrm{I}$-favorable provided $X$ is $\mathrm{I}$-favorable. And conversely, if a compactification of $X$ is $\mathrm{I}$-favorable, then so is $X$.
Because of that, very often when dealing with $\mathrm{I}$-favorable spaces, we can suppose that they are compact.

Let us introduced few more notations. Suppose $X\subset\mathbb I^A$ is a
compact space and $B\subset A$, where $\mathbb I=[0,1]$. Let $\pi_B\colon\mathbb
I^A\to\mathbb I^B$ be the natural projection and $p_B$ be
restriction map $\pi_B|X$. Let also $X_B=p_B(X)$. If $U\subset X$ we
write $B\in k(U)$ to denote that $p_{B}^{-1}\big(p_{B}(U)\big)=U$.
A base $\mathcal A$ for the topology of $X\subset\mathbb I^A$
consisting of open sets is called {\em special} if for every
finite $B\subset A$ the family $\{p_B(U):U\in\mathcal A, B\in
k(U)\}$ is a base for $p_B(X)$ and for each $U\in\mathcal A$ there is a finite set $B\subset A$ with $B\in k(U)$.

\begin{pro}
Let $X$ be a compact $\mathrm{I}$-favorable space and $w(X)=\tau$ is uncountable. Then there exists a
continuous inverse system $S=\{X_\delta, p^{\delta}_\gamma, \gamma<\delta<\lambda\}$, where $\lambda=\mathrm{cf}(\tau)$,
of compact $\mathrm{I}$-favorable spaces $X_\delta$ and skeletal bonding maps $p^{\delta}_\gamma$ such that
$w(X_\delta)<\tau$ for each $\delta<\lambda$ and
$X=\displaystyle\underleftarrow{\lim}S$.
\end{pro}

\begin{proof}
We embed $X$ in a Tychonoff cube $\mathbb
I^A$ with $|A|=\tau$ and fix a special open base $\mathcal A=\{U_\alpha:\alpha\in A\}$ for
$X$ of cardinality $\tau$ which consists of open sets such that
for each $\alpha$ there exists a finite set $H_\alpha\subset A$ with
$H_\alpha\in k(U_\alpha)$. Let $\sigma:\bigcup_{n\geq 0}\mathcal A^n\to\mathcal A$ be a winning
strategy. We represent $A$ as the union of an increasing transfinite family $\{A_\delta:\delta<\lambda\}$
with $|A_\delta|<\tau$, and
 let $\mathcal A_\delta=\{U_\alpha:\alpha\in A_\delta\}$ for each $\delta<\lambda$.

For any finite set $C\subset A$ let
$\gamma_C$ be a fixed countable base for $X_C$. Observe that for
every $U\in\mathcal A$ there exists a finite set $B(U)\subset A$
such that $B(U)\in k(U)$ and $p_{B(U)}(U)$ is open in
$X_{B(U)}$. We are going to construct
by transfinite induction increasing families $\{B_\delta:\delta<\lambda\}$ and $\{\mathcal B_\delta:\delta<\lambda\}\subset\mathcal A$
satisfying the following conditions for every $\delta<\lambda$:
\begin{itemize}
\item[(1)] $A_\delta\subset B_\delta\subset A$, $\mathcal A_\delta\subset\mathcal B_\delta$, $|B_\delta|=|\mathcal B_\delta|<\tau$;
\item[(2)]  $B_\delta\in k(U)$ for all $U\in\mathcal B_\delta$;
\item[(3)] $p_C^{-1}(\gamma_C)\subset\mathcal B_\delta$ for  each finite $C\subset B_\delta$;
\item[(4)]  $\sigma(U_1,..,U_n)\in\mathcal B_\delta$ for every finite family $\{U_1,..,U_n\}\subset\mathcal B_\delta$;
\item[(5)]  $B_\delta=\bigcup_{\gamma<\delta}B_\gamma$ and $\mathcal B_\delta=\bigcup_{\gamma<\delta}\mathcal B_\gamma$ for all limit
cardinals $\delta$.
\end{itemize}
Suppose all $B_\gamma$ and $\mathcal B_\gamma $, $\gamma<\delta$, have
already been constructed for some $\delta<\lambda$. If $\delta$ is a
limit cardinal, we put $B_\delta=\bigcup_{\gamma<\delta}B_\gamma$ and $\mathcal B_\delta=\bigcup_{\gamma<\delta}\mathcal B_\gamma$.
If $\delta=\gamma+1$, we construct by induction a sequence
$\{C(m)\}_{m\geq 0}$ of subsets of $A$, and a sequence $\{\mathcal
V_m\}_{m\geq 0}$ of subfamilies of $\mathcal A$ such that:
\begin{itemize}
\item $C_0=B_\gamma$ and $\mathcal V_0=\mathcal B_\gamma$;
\item $C(m+1)=C(m)\bigcup\{B(U):U\in\mathcal V_m\}$;
\item $\mathcal V_{2m+1}=\mathcal V_{2m}\bigcup\{\sigma(U_1,..,U_s): U_1,..,U_s\in\mathcal
V_{2m}, s\geq 1\}$;
\item $\mathcal V_{2m+2}=\mathcal
V_{2m+1}\bigcup\{p_C^{-1}(\gamma_C):C\subset C(2m+1){~}\mbox{is
finite}\}$.
\end{itemize}

Now, we define $B_\delta=\bigcup_{m\geq 0}C(m)$ and $\mathcal B_\delta=\bigcup_{m\geq 0}\mathcal V_m$. It is easily seen that
$B_\delta$ and $\mathcal B_\delta$ satisfy conditions (1)-(5).

For every $\delta<\lambda$ let $X_\delta=X_{B_\delta}$ and
$p_\delta=p_{B_\delta}$.  Moreover, if $\gamma<\delta$, we have
$B_\gamma\subset B_\delta$, and let
$p^\delta_\gamma=p^{B_\delta}_{B_\gamma}$. Since
$A=\bigcup_{\delta<\lambda}B_\delta$, we obtain a continuous inverse
system $S=\{X_\delta, p^{\delta}_\gamma, \gamma<\delta<\lambda\}$ whose limit is $X$.
Observe also that each $X_\delta$ is of weight $<\tau$ because
$p_\delta(\mathcal B_\delta)$ is a base for $X_\delta$ (see condition (3)).

\smallskip
{\em Claim $1.$ All bonding maps $p^\delta_\gamma$ are skeletal.}
\smallskip

It suffices to show that all $p_\delta$ are skeletal. And this is
really true because each family $\mathcal B_\delta$ is stable with
respect to $\sigma$, see (4). Hence, by \cite[Lemma 9]{kp1}, for
every open set $V\subset X$ there exists $W\in\mathcal B_\delta$
such that whenever $U\subset W$ and $U\in\mathcal B_\delta$ we have
$V\cap U\neq\varnothing$. The last statement yields that $p_\delta$
is skeletal. Indeed, let $V\subset X$ be open, and $W\in\mathcal
B_\delta$ be as above. Then $p_\delta(W)$ is open in
$X_\delta$ because of condition (2). We claim that
$p_\delta(W)\subset\overline{p_\delta(V)}$. Indeed, otherwise
$p_\delta(W)\backslash\overline{p_\delta(V)}$ would be a non-empty
open subset of $X_\delta$. So, $p_\delta(U)\subset
p_\delta(W)\backslash\overline{p_\delta(V)}$ for some $U\in\mathcal
B_\delta$ (recall that $p_\delta(\mathcal B_\delta)$ is a base for
$X_\delta$). Since, by (2), $p_\delta^{-1}(p_\delta(U))=U$ and
$p_\delta^{-1}(p_\delta(W))=W$, we obtain $U\subset W$ and $U\cap
V=\varnothing$ which is a contradiction.

Finally, since the class of $\mathrm{I}$-favorable spaces is closed with respect to skeletal images \cite[Lemma 1]{kp3},
all $X_\delta$ are $\mathrm{I}$-favorable.
\end{proof}

An inverse system
$S=\{X_\alpha, p^{\beta}_\alpha, \alpha<\beta<\tau\}$, where $\tau$ is a given cardinal, is said to be {\em almost continuous}
provided for every limit cardinal $\gamma$ the space $X_\gamma$ is the almost limit of the inverse system $S_\gamma=\{X_\alpha, p^{\beta}_\alpha, \alpha<\beta<\gamma\}$.
If $X=\displaystyle\mathrm{a}-\underleftarrow{\lim}S$ of an almost continuous inverse system $S$ and $H\subset X$, the set
$$q(H)=\{\alpha:\mathrm{Int}\big(\big((p^{\alpha+1}_{\alpha})^{-1}(\overline{p_{\alpha}(H)})\big)
\backslash\overline{p_{\alpha+1}(H)}\big)\neq\varnothing\}$$ is
called a {\em rank of $H$}.

\begin{lem}\cite[Lemma 3.1]{vv}
Let $X=\displaystyle\mathrm{a}-\underleftarrow{\lim}S$ and $U\subset
X$ be open, where $S=\{X_\alpha, p^{\beta}_\alpha,
\alpha<\beta<\tau\}$ is almost continuous inverse system with skeletal bonding
maps. Then we have:
\begin{itemize}
\item[(1)] $\alpha\not\in q(U)$ if and only if
$(p_\alpha^{\alpha+1})^{-1}\big(\mathrm{Int}\overline{p_\alpha(U)}\big)\subset\overline{p_{\alpha+1}(U)}$;
\item[(2)] $q(U)\cap[\alpha,\tau)=\varnothing$ provided
$U=p_\alpha^{-1}(V)$ for some open $V\subset X_\alpha$.
\end{itemize}
\end{lem}
\begin{lem}
Let $S=\{X_\alpha, p^{\beta}_\alpha, 1\leq\alpha<\beta<\tau\}$ be an almost continuous
inverse system with skeletal bonding maps and
$X=\displaystyle\mathrm{a}-\underleftarrow{\lim}S$. The the following hold for any open $U\subset X$:
\begin{itemize}
\item[(1)] If $(p^\alpha_1)^{-1}\big(\mathrm{Int}\overline{p_1(U)}\big)\subset\mathrm{Int}\overline{p_{\alpha}(U)}$
for all $\alpha<\tau$, then $p_1^{-1}\big(\mathrm{Int}\overline{p_1(U)}\big)\subset\overline{U}$;
\item[(2)] If $\lambda<\tau$ and $q(U)\cap [\lambda,\tau)=\varnothing$, then $p_\lambda^{-1}\big(\mathrm{Int}\overline{p_\lambda(U)}\big)\subset\mathrm{Int}\overline{U}$. 
    \end{itemize}
\end{lem}

\begin{proof}
The first item was proved in \cite[Lemma 3.2]{vv} under the assumption that $X=\displaystyle\underleftarrow{\lim}S$, but the same arguments work in our situation. Item (2) is equivalent to the inclusion
$(p_\lambda)^{-1}\big(\mathrm{Int}\overline{p_\lambda(U)}\big)\subset\overline{U}$.
 Let $A$ be the set of all $\alpha\in (\lambda,\tau)$ with
$(p^\alpha_\lambda)^{-1}\big(\mathrm{Int}\overline{p_\lambda(U)}\big)\setminus\overline{p_{\alpha}(U)}\neq\varnothing$.
Suppose $A$ is non-empty and let $\gamma=\min A$. Observe that $\gamma$ is a limit cardinal. Indeed, otherwise  $\gamma=\beta+1$ with
$\beta\geq\lambda$, so
$(p^\beta_\lambda)^{-1}\big(\mathrm{Int}\overline{p_\lambda(U)}\big)\subset\mathrm{Int}\overline{p_{\beta}(U)}$. Since $\beta\not\in q(U)$, according to
Lemma 2.3(1), we have $(p^\gamma_\beta)^{-1}\big(\mathrm{Int}\overline{p_\beta(U)}\big)\subset\overline{p_{\gamma}(U)}$. Hence,
$(p^\gamma_\lambda)^{-1}\big(\mathrm{Int}\overline{p_\lambda(U)}\big)\subset\overline{p_{\gamma}(U)}$, a contradiction.

Since $S$ is almost continuous and $\gamma$ is a limit cardinal, we have  $X_\gamma=\displaystyle\mathrm{a}-\underleftarrow{\lim}S_\gamma$, where $S_\gamma$ is the inverse system $\{X_\alpha, p^{\beta}_\alpha, \lambda\leq\alpha<\beta<\gamma\}$. Because $p_\gamma$ is skeletal, $U_\gamma=\mathrm{Int}\overline{p_\gamma(U)}\neq\varnothing$. So, we can apply item (1) to $X_\gamma$, the inverse system $S_\gamma$ and the open set $U_\gamma\subset X_\gamma$, to conclude that $(p^\gamma_\lambda)^{-1}\big(\mathrm{Int}\overline{p_\lambda(U)}\big)\subset\overline{p_{\gamma}(U)}$.
So, we obtain again a contradiction, which shows that $(p^\alpha_\lambda)^{-1}\big(\mathrm{Int}\overline{p_\lambda(U)}\big)\subset\overline{p_{\alpha}(U)}$ for all $\alpha\in [\lambda,\tau)$. Finally, because
  the system $\widetilde S_\lambda=\{X_\alpha, p^{\beta}_\alpha, \lambda\leq\alpha<\beta<\tau\}$ is almost continuous and
  $X=\displaystyle\mathrm{a}-\underleftarrow{\lim}\widetilde S_\lambda$, by item (1) we have  $p_\lambda^{-1}\big(\mathrm{Int}\overline{p_\lambda(U)}\big)\subset\mathrm{Int}\overline{U}$.
\end{proof}

Next lemma was established in \cite{vv} for continuous inverse systems. We present here a simplified proof concerning almost continuous systems.
\begin{lem}\cite[Lemma 3.3]{vv}
Let $S=\{X_\alpha, p^{\beta}_\alpha, \alpha<\beta<\tau\}$ be an almost
continuous inverse system with skeletal bonding maps and
$X=\displaystyle\mathrm{a}-\underleftarrow{\lim}S$. Assume $U, V\subset X$ are
open with $q(U)$ and $q(V)$ finite and
$\overline{U}\cap\overline{V}=\varnothing$. If $q(U)\cap q(V)\cap
[\gamma,\tau)=\varnothing$ for some $\gamma<\tau$, then
$\mathrm{Int}\overline{p_\gamma (U)}$ and
$\mathrm{Int}\overline{p_\gamma (V)}$ are disjoint.
\end{lem}

\begin{proof}
Suppose $\mathrm{Int} \overline{p_\gamma (U)}\cap
 \mathrm{Int}\overline{p_\gamma (V)}\neq\varnothing$. We are going to show by transfinite induction that $\mathrm{Int}
\overline{p_\beta(U)}\cap\mathrm{Int} \overline{p_\beta
(V)}\neq\varnothing$ for all $\beta\geq\gamma$. Assume this is done
for all $\beta\in(\gamma, \alpha)$
 with
 $\alpha<\tau$. If $\alpha$ is not a limit cardinal, then $\alpha-1$ belongs to at most one of the
 sets $q(U)$ and $q(V)$. Suppose $\alpha-1\not\in q(V)$.
 Hence,
 $(p^\alpha_{\alpha-1})^{-1}\big(\mathrm{Int}\overline{p_{\alpha-1}(V)}\big)\subset\mathrm{Int}\overline{p_{\alpha}(V)}$
 (see Lemma 2.3(1)). Due to our assumption, $\mathrm{Int}
\overline{p_{\alpha-1}(U)}\cap\mathrm{Int} \overline{p_{\alpha-1}
(V)}\neq\varnothing$. Moreover,
$p^\alpha_{\alpha-1}\big(\overline{p_\alpha(U)}\big)$ is dense in
$\overline{p_{\alpha-1}(U)}$. Hence, $\mathrm{Int}
\overline{p_{\alpha-1}(V)}$ meets
$p^\alpha_{\alpha-1}\big(\overline{p_\alpha(U)}\big)$. This yields
$\mathrm{Int}\overline{p_{\alpha}(V)}\cap\overline{p_\alpha(U)}\neq\varnothing$.
Finally, since $\overline{p_\alpha(U)}$ is the
closure of its interior,
$\mathrm{Int}\overline{p_{\alpha}(V)}\cap\mathrm{Int}\overline{p_\alpha(U)}\neq\varnothing$.

Suppose $\alpha>\gamma$ is a limit cardinal. Since $q(U)\cup q(V)$
is a finite set, there exists $\lambda\in(\gamma,\alpha)$ such that
$\beta\not\in q(U)\cup q(V)$ for all $\beta\in[\lambda,\alpha)$.
Now, we consider the almost continuous inverse system $S_\alpha=\{X_\delta, p^{\beta}_\delta, \lambda\leq\delta<\beta<\alpha\}$
with $X_\alpha=\displaystyle\mathrm{a}-\underleftarrow{\lim}S_\alpha$. Let $U_\alpha=\mathrm{Int}\overline{p_{\alpha}(U)}$ and
$V_\alpha=\mathrm{Int}\overline{p_{\alpha}(V)}$ and denote by $q_\alpha(U_\alpha)$ and $q_\alpha(V_\alpha)$ the ranks of $U_\alpha$ and $V_\alpha$
with respect to the system $S_\alpha$. The, according to Lemma 2.3(1), $\beta\in[\lambda,\alpha)$ does not belong to $q_\alpha(U_\alpha)$ if and only if
$(p^{\beta+1}_\beta)^{-1}\big(\mathrm{Int}\overline{p^\alpha_{\beta}(U_\alpha)}\subset\overline{p^\alpha_{\beta+1}(U_\alpha)}$. Since
$\overline{p^\alpha_{\beta}(U_\alpha)}=\overline{p_{\beta}(U)}$ and $\overline{p^\alpha_{\beta+1}(U_\alpha)}=\overline{p_{\beta+1}(U)}$, we obtain that
$\beta\not\in q_\alpha(U_\alpha)$ is equivalent to $\beta\not\in q(U)$. Similarly, $\beta\not\in q_\alpha(V_\alpha)$ iff $\beta\not\in q(V)$.
Consequently, $\beta\not\in q_\alpha(U_\alpha)\cup q_\alpha(V_\alpha)$ for all $\beta\in[\lambda,\alpha)$. Then, according to Lemma 2.4(2),
$(p^{\alpha}_\lambda)^{-1}\big(\mathrm{Int}\overline{p{\lambda}(U)}\subset\mathrm{Int}\overline{p_\alpha(U)}$ and
$(p^{\alpha}_\lambda)^{-1}\big(\mathrm{Int}\overline{p{\lambda}(V)}\subset\mathrm{Int}\overline{p_\alpha(V)}$. Because
$\mathrm{Int}\overline{p_\lambda(U)}\cap\mathrm{Int} \overline{p_\lambda(V)}\neq\varnothing$, we finally have
$\mathrm{Int}\overline{p_\alpha(U)}\cap\mathrm{Int} \overline{p_\alpha(V)}\neq\varnothing$. This completes the transfinite induction.

Therefore,
$\mathrm{Int}\overline{p_\beta(U)}\cap\mathrm{Int}\overline{p_\beta(V)}\neq\varnothing$
for all $\beta\in [\gamma,\tau)$. To finish the proof of this lemma,
take $\lambda(0)\in (\gamma,\tau)$ such that $\big(q(U)\cup
q(V)\big)\cap [\lambda(0),\tau)=\varnothing$. Then, according to Lemma 2.4(2) we have the following inclusions:
\begin{itemize}
\item $p_{\lambda(0)}^{-1}\big(\mathrm{Int}\overline{p_{\lambda(0)}(U)}\big)\subset\mathrm{Int}\overline{U}$;
\item $p_{\lambda(0)}^{-1}\big(\mathrm{Int}\overline{p_{\lambda(0)}(V)}\big)\subset\mathrm{Int}\overline{V}$.
\end{itemize}
Since
$\mathrm{Int}\overline{p_{\lambda(0)}(U)}\cap\mathrm{Int}\overline{p_{\lambda(0)}(V)}\neq\varnothing$,
the above inclusions imply
$\overline{U}\cap\overline{V}\neq\varnothing$, a contradiction.
Hence, $\mathrm{Int} \overline{p_\gamma (U)}\cap
 \mathrm{Int}\overline{p_\gamma (V)}=\varnothing$.
\end{proof}

Next proposition was announced in \cite[Proposition 3.2]{v} and a proof was presented in \cite[Proposition 3.4]{vv} (see Proposition 3.2
below for a similar statement concerning inverse systems with nearly open projections).
\begin{pro}\cite{v} Let $S=\{X_\alpha,p^{\beta}_\alpha, \alpha<\beta<\tau\}$ be an almost continuous
inverse system with skeletal bonding maps such that $X=\displaystyle\mathrm{a}-\underleftarrow{\lim}S$. Then the family of all open subsets of $X$ having a finite rank is a $\pi$-base for $X$.
\end{pro}

\begin{pro}
Let $X$ be a compact $\mathrm{I}$-favorable space. Then every embedding of $X$ in another space is $\pi$-regular.
\end{pro}

\begin{proof}
We are going to prove this proposition by transfinite induction with
respect to the weight $w(X)$. This is true if $X$ is metrizable, see
for example \cite[\S21, XI, Theorem 2]{k}. Assume the proposition is
true for any compact $\mathrm{I}$-favorable space $Y$ of weight $<\tau$, where $\tau$ is an uncountable cardinal. Suppose $X$ is compact
$\mathrm{I}$-favorable with $w(X)=\tau$.
Then, by Proposition 2.2, $X$ is the limit space of a continuous inverse system $\displaystyle S=\{X_\alpha, p^{\beta}_\alpha,
\alpha<\beta<\lambda\}$, where $\lambda=\mathrm{cf}(\tau)$, such that all $X_\alpha$ are compact $\mathrm{I}$-favorable spaces of weight
$<\tau$ and all bonding maps are surjective and skeletal. If suffices to show that there exists a $\pi$-regular embedding of $X$
in a Tychonoff cube $\I^A$ for some set $A$.

By Proposition 2.6, $X$ has a $\pi$-base $\mathcal B$ consisting of
open sets $U\subset X$  with finite rank. For every $U\in\mathcal B$
let $\Omega(U)=\{\alpha_0, \alpha, \alpha+1:\alpha\in q(U)\}$, where
$\alpha_0<\lambda$ is fixed. Obviously, $X$ is a subset of
$\prod\{X_\alpha:\alpha< \lambda\}$. For every $U\in\mathcal B$ we
consider the open set $\Gamma(U)\subset\prod\{X_\alpha:\alpha<\lambda\}$ defined by
$\Gamma(U)=\prod\{\mathrm{Int}\overline{p_\alpha(U)}:\alpha\in
\Omega(U)\}\times\prod\{X_\alpha:\alpha\not\in\Omega(U)\}.$

\smallskip
{\em Claim $2$. $\Gamma (U_1)\cap\Gamma (U_2)=\varnothing$ whenever
$\overline{U_1}\cap\overline{U_2}=\varnothing$. Moreover, there exists $\beta\in\Omega(U_1)\cap\Omega(U_2)$ with $\overline{p_\beta(U_1)}\cap
 \overline{p_\beta(U_2)}=\varnothing$.}

\smallskip
Let $\beta=\max\{\Omega(U_1)\cap \Omega(U_2)\}$. Then $\beta$ is either $\alpha_0$ or $\max\{q(U_1)\cap q(U_2)\}+1$. In both
cases $q(U_1)\cap q(U_2)\cap [\beta,\lambda)=\varnothing$. According to Lemma 2.5, $\mathrm{Int}\overline{p_\beta(U_1)}\cap\mathrm{Int}\overline{p_\beta(U_2)}=\varnothing$.
 Since $\beta\in \Omega(U_1)\cap \Omega(U_2)$, $\Gamma(U_1)\cap\Gamma(U_2)=\varnothing$.

For every $U\in\mathcal B$ and $\alpha$ let
$U_\alpha=\mathrm{Int}\overline{p_\alpha(U)}$.

\smallskip
{\em Claim $3$. $\bigcap_{\alpha\in\Delta}p_\alpha^{-1}(V_\alpha)\cap U\neq\varnothing$ for every finite set
$\Delta\subset\{\alpha:\alpha<\lambda\}$, where each $V_\alpha$ is an open and dense subset of $U_\alpha$.}

Obviously, this is true if $|\Delta|=1$. Suppose it is true for all $\Delta$ with $|\Delta|\leq n$ for some $n$, and let
$\{\alpha_1,..,\alpha_n, \alpha_{n+1}\}$ be a finite set of $n+1$ cardinals $<\tau$. Then $\displaystyle V=\bigcap_{i\leq n}p_{\alpha_i}^{-1}(V_{\alpha_i})\cap U\neq\varnothing$. Since $p_{\alpha_{n+1}}$ is a closed and skeletal map, $\displaystyle W=\mathrm{Int}\overline{p_{\alpha_{n+1}}(V)}$ is a non-empty subset of
$\displaystyle X_{\alpha_{n+1}}$ and $W\subset U_{\alpha_{n+1}}$.
Consequently $V_{\alpha_{n+1}}\cap W\neq\varnothing$. So, $V_{\alpha_{n+1}}\cap p_{\alpha_{n+1}}(V)\neq\varnothing$ and $\displaystyle \bigcap_{i\leq n+1}
p_{\alpha_i}^{-1}(V_{\alpha_i})\cap U\neq\varnothing$.

\smallskip
{\em Claim $4$. $\Gamma (U)\cap X$ is a non-empty subset of
$\overline{U}$ for all $U\in\mathcal B$.}

\smallskip
We are going to show first that $\Gamma (U)\cap X\neq\varnothing$ for all $U\in\mathcal B$. Indeed, we fix such $U$  and let
$\Omega(U)=\{\alpha_i:i\leq k\}$ with $\alpha_i\leq\alpha_j$ for $i\leq j$. By Claim 3, there exists $\displaystyle x\in
\bigcap_{i\leq k} p_{\alpha_i}^{-1}(U_{\alpha_i})\cap U$. So, $p_{\alpha_i}(x)\in U_{\alpha_i}$ for all $i\leq k$. This implies
$\Gamma (U)\cap X\neq\varnothing$.
To show that $\Gamma(U)\cap X\subset\overline{U}$, let
$y\in\Gamma(U)\cap X$ and $\beta(U)=\max q(U)+1$. Then
$p_{\beta(U)}(y)\in\mathrm{Int}\overline{p_{\beta(U)}(U)}$.
 Since $\alpha\not\in q(U)$ for all $\alpha\geq\beta(U)$, according to Lemma 2.4(2), we have
$y\in p_{\beta(U)}^{-1}\big(\mathrm{Int}\overline
{p_{\beta(U)}(U)}\big)\subset\overline{U}$. This completes the proof of Claim 4.

According to our assumption, each $X_\alpha$ is $\pi$-regularly
embedded in $\I^{A(\alpha)}$ for some $A(\alpha)$. So, there exists
a $\pi$-regular operator $\reg_\alpha:\mathcal
T_{X_\alpha}\to\mathcal T_{\I^{A(\alpha)}}$. For every $U\in\mathcal
B$ define the open set $\theta_1(U)\subset\prod_{\alpha<\lambda}\I^{A(\alpha)}$,
$$\theta_1(U)=\prod_{\alpha\in\Omega(U)}\reg_\alpha\big(\mathrm{Int}\overline{p_\alpha(U)}\big)\times\prod_{\alpha\not\in\Omega(U)}\I^{A(\alpha)}.$$
Now, we define a function $\theta$ from $\Tee_X$ to the topology of $\prod_{\alpha<\lambda}\I^{A(\alpha)}$ by
$$\theta(G)=\bigcup\{\theta_1(U):U\in\mathcal B\hbox{~}\mbox{and}\hbox{~}\overline{U}\subset G\}.$$
Let show
that $\theta$ is $\pi$-regular. It follows from Claim 2 that
$\theta(G_1)\cap\theta(G_2)=\varnothing$ provided $G_1\cap
G_2=\varnothing$. On the other hand, for every open $G\subset X$ we have $\theta(G)\cap
X\subset\bigcup\{\Gamma(U)\cap X:U\in\mathcal
B\hbox{~}\mbox{and}\hbox{~}\overline{U}\subset G\}$. Hence, by Claim 4, $\theta(G)\cap
X\subset\bigcup\{\overline U:U\in\mathcal
B\hbox{~}\mbox{and}\hbox{~}\overline{U}\subset G\}\subset G$.
To prove that $\theta(G)\cap X$ a dense subset of $G$ it suffices to show that $\theta_1(U)\cap X\neq\varnothing$ for all $U\in\mathcal B$ with $\overline U\subset G$. To this end, we fix such $U$ and let
 $V_\alpha=\reg_\alpha(U_\alpha)\cap X_\alpha$ for every $\alpha\in\Omega(U)$. Then $V_\alpha$ is a dense open subset of $U_\alpha$,
and by Claim 3, $V=\bigcap_{\alpha\in\Omega(U)}p_\alpha^{-1}(V_\alpha)\cap U$ is a non-empty subset of $\theta_1(U)\cap X$.
Therefore, $X$ is $\pi$-regularly embedded in $\I^A=\prod_{\alpha<\lambda}\I^{A(\alpha)}$.
\end{proof}

Next proposition was established in \cite{vv} (Proposition 3.7) assuming that $X$ is a $\pi$-regularly $C^*$-embedded
subset of the limit space of a $\sigma$-complete inverse system with open bounding maps and second countable spaces. The arguments there
work if $X$ is just a $\pi$-regularly embedded subset of a product of second countable spaces.
\begin{pro}
Let $X$ be a  $\pi$-regularly embedded subspace of a product of second countable spaces. Then $X$ is
skeletally generated.
\end{pro}

{\em Proof of Theorem $1.1$}. To prove implication $(1)\Rightarrow
(2)$, suppose $X$ is $\mathrm{I}$-favorable
subspace of a space $Y$. Then
$\widetilde X=\overline{X}^{\beta Y}$ is a compactification of $X$. Since $\widetilde X$
is also $\mathrm{I}$-favorable,
according to Proposition 2.7, $\widetilde X$ is
$\pi$-regularly embedded in $\beta Y$. This yields that $X$ is
$\pi$-regularly embedded in $Y$.

$(2)\Rightarrow (3)$ Let $X$ be a  subset of a Tychonoff cube
$\I^A$. Then $X$ is $\pi$-regularly embedded in $\I^A$, and by Proposition 2.8,
$X$ is skeletally generated.

The implication $(3)\Rightarrow (1)$ follows as follows. If $X$ is skeletally generated, then $X=\mathrm{a}-\displaystyle\lim_\leftarrow S$, where
$S$ is an almost $\sigma$-continuous inverse system of second countable spaces $X_\alpha$, $\alpha\in A$, and skeletal bounding maps $p^\alpha_\beta$.
Because each $X_\alpha$ is $\mathrm I$-favorable, it follows from \cite[Theorem 3.3]{ku1} (see also \cite[Theorem 13]{kp1}) that $X$ is  $\mathrm I$-favorable too. {$\Box$}

\smallskip
{\em Proof of Corollary $1.2$.} Suppose $X$ is an
$\mathrm I$-favorable subspace of an extremally disconnected space $Y$. Then there exists a $\pi$-regular operator
$\mathrm{e}\colon\mathcal T_X\to\mathcal T_Y$. We need to show that the closure (in $X$) of every open subset of $X$ is also open. Since
$Y$ is extremally disconnected, $\overline{\mathrm{e}(U)}^Y$ is open in $Y$. So, the proof will be done if we prove that
$\overline{\mathrm{e}(U)}^Y\cap X=\overline{U}^X$ for all $U\in\mathcal T_X$. Because $\mathrm{e}(U)\cap X$ is a dense subset of $U$, we have
$\overline{U}^X\subset\overline{\mathrm{e}(U)}^Y\cap X$. Assume $\overline{\mathrm{e}(U)}^Y\cap X\backslash\overline{U}^X\neq\varnothing$ and choose $V\in\mathcal T_X$ with $V\subset\overline{\mathrm{e}(U)}^Y\backslash\overline{U}^X$. Then
$\mathrm{e}(V)\cap\overline{\mathrm{e}(U)}^Y\neq\varnothing$, so
$\mathrm{e}(V)\cap\mathrm{e}(U)\neq\varnothing$. The last one
contradicts $U\cap V=\varnothing$. \hfill{$\Box$}

\smallskip
{\em Proof of Corollary $1.3$.} Suppose $X$ is $\mathrm{I}$-favorable and $W\subset X$ is open. Then there is a $\pi$-regular  embedding of $X$ into a product $\Pi$ of lines. Obviously, $W$ is also $\pi$-regularly embedded in $\Pi$, and by Proposition 2.8, $W$ is $\mathrm{I}$-favorable. \hfill{$\Box$}

\section{Quasi $\kappa$-metrizable  spaces}
{\em Proof of Theorem $1.4$.} Suppose $X$ is a compact $\mathrm I$-favorable.  We embed $X$ in $\mathbb R^\tau$ for some cardinal $\tau$, and let
$\rho(z, C)$ be a $\kappa$-metric on $\mathbb R^\tau$, see \cite{sc1}.
According to Theorem 1.1, there exists a $\pi$-regular function
$\mathrm e:\mathcal T_X\to\mathcal T_{\mathbb R^\tau}$. We define a new function $\mathrm e_1:\mathcal T_X\to\mathcal T_{\mathbb R^\tau}$,
$$\mathrm e_1(U)=\bigcup\{\mathrm e(V):V\in\mathcal T_X{~}\mbox{and}{~}\overline V\subset U\}.$$
Obviously $\mathrm e_1$ is $\pi$-regular and it is also monotone, i.e. $U\subset V$ implies
$\mathrm e_1(U)\subset\mathrm e_1(V)$. Moreover, for every increasing transfinite family $\gamma=\{U_\alpha\}$ of open sets in $Y$ we have
$\mathrm e_1(\bigcup_\alpha U_\alpha)=\bigcup_\alpha\mathrm e_1(U_\alpha)$. Indeed, if $z\in\mathrm e_1(\bigcup_\alpha U_\alpha)$, then there is an open set
$V\in\mathcal T_X$ with $\overline V\subset\bigcup_\alpha U_\alpha$ and $z\in\mathrm e(V)$. Since $\overline V$ is compact and the family is increasing,
$\overline V$ is contained in some $U_{\alpha_0}$. Hence, $z\in\mathrm e(V)\subset\mathrm e_1(U_{\alpha_0})$. Consequently,
$\mathrm e_1(\bigcup_\alpha U_\alpha)\subset\bigcup_\alpha\mathrm e_1(U_\alpha)$. The other inclusion follows from monotonicity of $\mathrm e_1$.

Now, for every open $U\subset X$ and $x\in X$ we can define the function $d(x,\overline U)=\rho(x,\overline{\mathrm e_1(U)})$, where $\overline{\mathrm e_1(U)}$ is the closure of $\mathrm e_1(U)$ in $\mathbb R^\tau$. It is easily seen that
$d(x,\overline U)$ satisfies axioms $K2) - K3)$. Let show that it also satisfies $K4)$ and $K1^*)$. Indeed, assume $\{C_\alpha\}$ is an increasing transfinite family of regularly closet sets in $X$. We put $U_\alpha=\mathrm{Int}C_\alpha$ for every $\alpha$ and $U=\bigcup_\alpha U_\alpha$.  Thus, $\mathrm e_1(U)=\bigcup_\alpha\mathrm e_1(U_\alpha)$. Since $\{\overline{\mathrm e_1(U_\alpha)}\}$ is an increasing transfinite family of regularly closed sets in $\mathbb R^\tau$, $$d(x,\overline{\bigcup_\alpha C_\alpha})=\rho(x,\overline{\bigcup_\alpha\mathrm e_1(U_\alpha)})=\inf_\alpha\rho(x,\overline{\mathrm e_1(U_\alpha)})=\inf_\alpha d(x,C_\alpha).$$
To show that $K1^*)$ also holds, observe that $d(x,\overline{U})=0$ if and only if
$x\in X\cap\overline{\mathrm e_1(U)}$. Thus, we need to show that there is an open dense subset $V$ of $X\setminus\overline U$ such that
$X\cap\overline{\mathrm e_1(U)}=X\setminus V$. Because $\mathrm e_1(U)\cap X$ is dense in $U$, $\overline U\subset\overline{\mathrm e_1(U)}$. Hence,
$V=X\setminus\overline{\mathrm e_1(U)}$ is contained in $X\setminus\overline U$. To prove $V$ is dense in $X\setminus\overline U$, let $x\in X\setminus\overline U$ and $W_x\subset X\setminus\overline U$ be an open neighborhood of $x$. Then $W\cap U$ is empty, so $\mathrm e_1(W)\cap\mathrm e_1(U)=\varnothing$. This yields
$\mathrm e_1(W)\cap X\subset V$. On the other hand, $\mathrm e_1(W)\cap X$ is a non-empty subset of $W$, hence $W\cap V\neq\varnothing$. Therefore, $d$ is an quasi $\kappa$-metric on $X$.

Suppose $X$ is a compact space and let $d(x,\overline U)$ be a quasi $\kappa$-metric on $X$. We are going to show that $X$ is skeletally generated. To this end we embed $X$ in $\mathbb I^A$ for some $A$. Following the notations from the proof of Proposition 2.2, for any countable set $B\subset A$ let $\mathcal A_B$ be the countable base for $X_B=p_B(X)$ consisting of all open sets in $X_B$ of the form $X_B\cap\prod_{\alpha\in B} V_\alpha$, where each $V\alpha$ is an open subinterval of $\mathbb I=[0,1]$ with rational end-points and $V_\alpha\neq\mathbb I$ for finitely many $\alpha$.
For any open $U\subset X$ denote by $f_U$ the function $d(x,\overline U)$. We also write $p_B\prec g$, where $g$ is a map defined on $X$, if there is a map $h:p_B(X)\to g(X)$ such that $g=h\circ p_B$. Since $X$ is compact this is equivalent to the following: if
$p_B(x_1)=p_B(x_2)$ for some $x_1,x_2\in X$, then $g(x_1)=g(x_2)$.  We say that a countable set $B\subset A$ is {\em $d$-admissible} if $p_B\prec f_{p_B^{-1}(V)}$ for every $V\in\mathcal A_B$. Denote by $\mathcal D$ the family of all $d$-admissible subsets of $A$. We are going to show that all maps $p_B:X\to X_B$, $B\in\mathcal D$, are skeletal and the inverse system $S=\{X_B: p^B_C:C\subset B, C,B\in\mathcal D\}$ is $\sigma$-continuous with $\displaystyle X=\lim_\leftarrow S$.

\smallskip
{\em Claim $5$. For every countable set $C\subset A$ there is $B\in\mathcal D$ with $C\subset B$.}

We are going to construct a sequence of countable sets $B_n\subset A$ such that for every $n\geq 1$ we have:
\begin{itemize}
\item $C\subset B_n\subset B_{n+1}$;
\item $p_{B_{n+1}}\prec f_{p_{B_n}^{-1}}(V)$ for all $V\in\mathcal A_{B_n}$.
\end{itemize}
We show the construction of $B_1$, the other sets $B_n$ can be obtained in a similar way. Every function $f_{{p_C}^{-1}}(V)$, $V\in\mathcal A_C$, has a continuous extension $\widetilde f_{{p_C}^{-1}}(V)$ on $\mathbb I^A$. Moreover, every continuous function $g$ on $\mathbb I^A$ depends on countably many coordinates (i.e., there exists a countable set $B_g\subset A$ with $\pi_{B_g}\prec g$). This fact allows us to find a countable set $B_1\subset A$ containing $C$ such that
$p_{B_1}\prec f_{{p_C}^{-1}}(V)$ for all $V\in\mathcal A_C$. Next, let $B=\bigcup_{n=1}B_n$. Since $\mathcal A_B$ is the union of all families
$\{(p^B_{B_n})^{-1}(V):V\in\mathcal A_{B_n}\}$, $n\geq 1$, for every $W\in\mathcal A_B$ there is $m$ and $V\in\mathcal A_{B_m}$ with
$p_B^{-1}(W)=p_{B_m}^{-1}(V)$. Then, according to the construction of the sets $B_n$, we have $p_{B_{m+1}}\prec f_{p_B^{-1}(W)}$. Hence $p_B\prec f_{p_B^{-1}(W)}$ for all $W\in\mathcal A_B$, which means
that $B$ is $d$-admissible.

\smallskip
{\em Claim $6$. For every $B\in\mathcal D$ the map $p_B$ is skeletal.}

Suppose there is an open set $U\subset X$ such that the interior in $X_B$ of the closure $\overline{p_B(U)}$ is empty. Then $W=X_B\setminus\overline{p_B(U)}$
is dense in $X_B$. Let $\{W_m\}_{m\geq 1}$ be a countable cover of $W$ with $W_m\in\mathcal A_B$ for all $m$. Since $\mathcal A_B$ is finitely additive, we may assume that $W_m\subset W_{m+1}$, $m\geq 1$. Because $B$ is $d$-admissible, $p_B\prec f_{p_B^{-1}(W_m)}$ for all $m$. Hence, there are continuous functions
$h_m:X_B\to\mathbb R$ with $f_{p_B^{-1}(W_m)}=h_m\circ p_B$, $m\geq 1$. Recall that $f_{p_B^{-1}(W_m)}(x)=d(x,\overline{p_B^{-1}(W_m)})$ and
$p_B^{-1}(W)=\bigcup_{m\geq 1}p_B^{-1}(W_m)$. Therefore, $f_{p_B^{-1}(W)}(x)=d(x,\overline{p_B^{-1}(W)})=\inf_m f_{p_B^{-1}(W_m)}(x)$ for all $x\in X$. Moreover,
$f_{p_B^{-1}(W_{m+1})}(x)\leq f_{p_B^{-1}(W_m)}(x)$ because $W_m\subset W_{m+1}$. The last inequalities together with $p_B\prec f_{p_B^{-1}(W_m)}$
yields that $p_B\prec f_{p_B^{-1}(W)}$. So, there exists a continuous function $h$ on $X_B$ with $d(x,\overline{p_B^{-1}(W)})=h(p_B(x))$ for all $x\in X$.
Since $p_B(\overline{p_B^{-1}(W)})=\overline W=X_B$, we have that $h$ is the constant function zero. Then $d(x,\overline{p_B^{-1}(W)})=0$  for all $x\in X$.
But $\overline{p_B^{-1}(W)}\cap U=\varnothing$. So, according to $K1^*)$, there is a dense open subset $U'$ of $U$ with $d(x,\overline{p_B^{-1}(W)})>0$ for each $x\in U'$, a contradiction.

It is easily seen that the union of any increasing sequence of $d$-admissible sets is also $d$-admissible. This fact and Claims 5 yield that the
inverse system $S=\{X_B: p^B_C:C\subset B, C,B\in\mathcal D\}$ is $\sigma$-continuous and $\displaystyle X=\lim_\leftarrow S$. Finally, by Claim 6,
all maps $p_B$, $B\in\mathcal D$, are skeletal. So are the bounding maps $p^B_C$ in $S$. Therefore, $X$ is skeletally generated, and hence $\mathrm I$-favorable by Theorem 1.1.

{\em Proof of Corollary $1.5$.}
Since $Y=\beta X$ is $\mathrm I$-favorable, by Theorem 1.4 there is a quasi $\kappa$-metric $d$ on $Y$.
We are going to show that $d_X(x,\overline U^X)=d(x,\overline U)$, $U\in\mathcal T_X$, defines a quasi $\kappa$-metric on $X$, where $\overline U^X$  and $\overline U$ is the closure of $U$ in $X$ and $Y$ respectively. Since $\overline U$ is regularly closed in $Y$, this definition is correct. It follows directly from the definition that $d_X$ satisfies axioms $K2)$ and $K3)$. Because for any increasing transfinite family $\{C_\alpha\}$ of regularly closed sets in $X$ the family $\{\overline{C_\alpha}\}$ is also increasing and consists of regularly closed sets in $Y$,
$$d_X(x,\overline{\bigcup_\alpha C_\alpha}^X)=d(x,\overline{\bigcup_\alpha C_\alpha})=\inf_\alpha d(x,\overline{C_\alpha})=\inf_\alpha d_X(x,C_\alpha),$$ $d_X$ satisfies $K4)$. Finally, $d_X$ satisfies also $K1^*)$. Indeed, for any $U\in\mathcal T_X$ there exists $V\in\mathcal T_Y$ such that $V$ is dense in $Y\setminus\overline U$ and
$d(x,\overline U)>0$ if and only if $x\in V$. This implies that the set $V\cap X$ is dense in $X\setminus\overline U^X$ and $d_X(x,\overline U^X)>0$ iff
$x\in V\cap X$. So, $d_X$ is a quasi $\kappa$-metric on $X$.

\section{Inverse systems with nearly open bounding maps}
In this section we consider almost continuous inverse systems with nearly open bounding maps. Recall that a map $f:X\to Y$ is nearly open \cite{ArTk}
if $f(U)\subset\mathrm{Int}\overline{f(U)}$  for every open $U\subset X$. Nearly open maps were considered by Tkachenko \cite{tk} under the name $d$-open maps.
The following properties of
ranks were established in Lemmas 2.3-2.5 when consider almost continuous inverse systems with skeletal bounding maps. The same
proofs remain valid and for inverse systems with nearly open bounding maps.
\begin{lem}
Let $X=\displaystyle\mathrm{a}-\underleftarrow{\lim}S$,
where $S=\{X_\alpha, p^{\beta}_\alpha,
\alpha<\beta<\tau\}$ is almost continuous with nearly open bonding
maps. Then for every open sets $U,V\subset X$ we have:
\begin{itemize}
\item[(1)] $\alpha\not\in q(U)$ if and only if
$(p_\alpha^{\alpha+1})^{-1}\big(\mathrm{Int}\overline{p_\alpha(U)}\big)\subset\overline{p_{\alpha+1}(U)}$;
\item[(2)] $q(U)\cap[\alpha,\tau)=\varnothing$ provided
$U=p_\alpha^{-1}(W)$ for some open $W\subset X_\alpha$;
\item[(3)] Suppose $q(U)$ and $q(V)$ are finite and
$\overline{U}\cap\overline{V}=\varnothing$. If $q(U)\cap q(V)\cap
[\gamma,\tau)=\varnothing$ for some $\gamma<\tau$, then
$\mathrm{Int}\overline{p_\gamma (U)}$ and
$\mathrm{Int}\overline{p_\gamma (V)}$ are disjoint.
\end{itemize}
\end{lem}

Next proposition was announced in \cite[Proposition 2.2]{v} without a proof.
Note that a similar statement was established in \cite{sc1} for inverse systems with open bounding maps.

\begin{pro}\cite{v} Let $S=\{X_\alpha,
p^{\beta}_\alpha, \alpha<\beta<\tau\}$ be an almost continuous
inverse system with nearly open bonding maps such that
$X=\displaystyle\mathrm{a}-\underleftarrow{\lim}S$. Then the family
of all open subsets of $X$ having a finite rank is a base for
$X$.
\end{pro}

\begin{proof}
We are
going to show by transfinite induction that for every $\alpha<\tau$
the open subsets $U\subset X$ with $q(U)\cap[1,\alpha]$ being finite
form a base for $X$. Obviously, this is true for finite
$\alpha$, and it holds for $\alpha+1$ provided it is true for
$\alpha$. So, it remains to prove this statement for a limit
cardinal $\alpha$ if it is true for any $\beta<\alpha$. Suppose
$G\subset X$ is open and $x\in G$. Since $p_\alpha$ is nearly open, $G_\alpha=\mathrm{Int}\overline{p_\alpha(G)}$
contains $p_\alpha(G)$ (here both interior and closure are taken in $X_\alpha$).
Let $S_\alpha=\{X_\gamma, p^{\beta}_\gamma,
\gamma<\beta<\alpha\}$, $Y_\alpha=\displaystyle\underleftarrow{\lim}
S_\alpha$ and $\widetilde{p}^{\alpha}_\gamma\colon Y_\alpha\to X_\gamma$
are the limit projections of $S_\alpha$. Obviously, $X_\alpha$ is
naturally embedded as a dense subset of $Y_\alpha$ and each
$\widetilde{p}^{\alpha}_\gamma$ restricted on $X_\alpha$ is
$p^{\alpha}_\gamma$.
So, there exists $\gamma<\alpha$ and an open set $U_\gamma\subset
X_\gamma$ containing $x_\gamma=p_\gamma(x)$ such that $\displaystyle
(\widetilde{p}^{\alpha}_\gamma)^{-1}(U_\gamma)\subset\mathrm{Int}_{Y_\alpha}\overline{G_\alpha}^{Y_\alpha}$.
Consequently, $\displaystyle
(p^{\alpha}_\gamma)^{-1}(U_\gamma)\subset G_\alpha$.
We can suppose that $U_\gamma=\mathrm{Int}\overline{U_\gamma}$.
Then, according to the inductive assumption, there is an open set $W\subset X$ such that $q(W)\cap [1,\gamma]$ is finite
and $x\in W\subset p_\gamma^{-1}(U_\gamma)\cap G$.
 So,
$x_\gamma\in p_\gamma(W)\subset W_\gamma=\mathrm{Int}\overline{p_\gamma(W)}$ and
 $W_\gamma\subset U_\gamma$. Hence, $x\in p_\gamma^{-1}(W_\gamma)\cap G\subset G$.
 Next claim completes the induction.

\smallskip
{\em Claim $7$. $q\big(p_\gamma^{-1}(W_\gamma)\cap G\big)\cap
[1,\alpha]=q(W)\cap [1,\gamma]$.}
\smallskip

Indeed, for every $\beta\leq\gamma$ we have
$\overline{p_\beta\big(p_\gamma^{-1}(W_\gamma)\cap
G\big)}=\overline{p_\beta(W)}$. This implies
$$q(W)\cap
[1,\gamma]=q\big(p_\gamma^{-1}(W_\gamma)\cap G\big)\cap [1,\gamma].\leqno{(6)}$$
Moreover, since $\displaystyle
(p^{\alpha}_\gamma)^{-1}(W_\gamma)\subset\displaystyle
(p^{\alpha}_\gamma)^{-1}(U_\gamma)\subset\overline{p_\alpha(G)}$, we have
$$\overline{p_\beta\big(p_\gamma^{-1}(W_\gamma)\cap
G\big)}=\overline{p_\beta\big(p_\gamma^{-1}(W_\gamma)\big)}$$ for each $\beta\in [\gamma,\alpha]$.
Hence, $$q\big(p_\gamma^{-1}(W_\gamma)\cap G\big)\cap
[\gamma,\alpha]=q\big(p_\gamma^{-1}(W_\gamma)\big)\cap
[\gamma,\alpha].\leqno{(7)}$$
Note that, by Lemma 4.1(2),
$q\big(p_\gamma^{-1}(W_\gamma)\big)\cap
[\gamma,\alpha]=\varnothing$. Then the combination of $(1)$ and
$(2)$ provides the proof of the claim.

Therefore, for every $\alpha<\tau$ the open sets $W\subset X$ with
$q(W)\cap [1,\alpha]$ being finite form a base for $X$. Now, we can
finish the proof of the proposition. If $V\subset X$ is open and $x\in V$ we find
a set $G\subset V$ with $x\in G=p_\beta^{-1}(G_\beta)$, where $G_\beta$
is open in $X_\beta$ for some $\beta<\tau$. Then there exists an open set $W\subset G$ containing $x$
such that $q(W)\cap [1,\beta]$ is finite. Let
$W_\beta=\mathrm{Int}\overline{p_\beta(W)}$ and
$U=p_\beta^{-1}(W_\beta\cap G_\beta)$. It is easily seen that $x\in U$ and
$\overline{p_\nu(U)}=\overline{p_\nu(W)}$ for all $\nu\leq\beta$.
This yields  $q(U)\cap [1,\beta]=q(W)\cap [1,\beta]$. On the
other hand, by Lemma 4.1(2), $q(U)\cap [\beta,\tau)=\varnothing$.
Hence $U$ is a neighborhood of $x$ which is contained in $V$ and $q(U)$ is finite.
\end{proof}

Similar to the previous proposition, next one was also announced in \cite[Proposition 2.3]{v} without a proof.
\begin{pro}\cite{v}
Let $S=\{X_\alpha,
p^{\beta}_\alpha, \alpha<\beta<\tau\}$ be an almost continuous
inverse system with nearly open bonding maps such that
$X=\displaystyle\mathrm{a}-\underleftarrow{\lim}S$. Then:
\begin{itemize}
\item[(1)] $X$ is regularly embedded in $\prod_{\alpha<\tau}X_\alpha$;
\item[(2)] If, additionally, each $X_\alpha$ is regularly embedded in a space $Y_\alpha$, then $X$ is regularly embedded in $\prod_{\alpha<\tau}Y_\alpha$.
\end{itemize}
\end{pro}

\begin{proof}
$(1)$ We consider the embedding of $X$ in $\widetilde{X}=\prod_{\alpha<\tau}X_\alpha$ generated by the maps $p_\alpha$. According to Proposition 4.2, $X$ has a base $\mathcal B$ consisting of
open sets $U\subset X$  with finite rank $q(U)$. As in Proposition 2.7, for every $U\in\mathcal B$
let $\Omega(U)=\{\alpha_0, \alpha, \alpha+1:\alpha\in q(U)\}$, where
$\alpha_0<\tau$ is fixed.  For all $U\in\mathcal B$ and $\alpha<\tau$ let $U_\alpha=\mathrm{Int}\overline{p_\alpha(U)}$
and  $\Gamma(U)\subset\prod\{X_\alpha:\alpha<
\tau\}$ be defined by
$$\Gamma(U)=\prod\{U_\alpha:\alpha\in
\Omega(U)\}\times\prod\{X_\alpha:\alpha\not\in\Omega(U)\}.$$ Since $p_\alpha(U)\subset U_\alpha$ for each $\alpha$,
$U$ is contained in $\Gamma(U)$.

Using the arguments from the proof of Proposition 2.7, one can show that $\Gamma(U)\cap X\subset\overline U$.
Finally, we define the required regular operator $\mathrm{e}:\Tee_X\to\Tee_{\widetilde{X}}$ by
$\mathrm{e}(V)=\bigcup\{\Gamma(U):U\in\mathcal B,\overline U\subset V\}$.

$(2)$ For each $\alpha<\tau$ let $\mathrm{e}_\alpha:\Tee_{X_\alpha}\to\Tee_{Y_\alpha}$ be a regular operator.
Define a function $\theta_1:\mathcal B\to\Tee_{\widetilde Y}$, where $\widetilde Y=\prod_{\alpha<\tau}Y_\alpha$, by
$$\theta_1(U)=\prod_{\alpha\not\in \Omega(U)}\mathrm{e}_\alpha(U_\alpha)\times\prod_{\alpha\not\in \Omega(U)}Y_\alpha.$$
Consider  $\theta:\Tee_X\to\Tee_{\widetilde Y}$,
$\theta(G)=\bigcup\{\theta_1(U):U\in\mathcal
B\hbox{~}\mbox{and}\hbox{~}\overline{U}\subset G\}.$ Since $\theta_1(U)\cap X=\Gamma (U)$ and $U\subset\Gamma(U)\subset\overline U$
for any $U\in\mathcal B$, $\theta(G)\cap X=G$. Moreover, Claim 4 implies that $\theta(G_1)\cap\theta(G_2)=\varnothing$ provided
$G_1\cap G_2=\varnothing$. Thus, $\theta$ is a regular operator.
\end{proof}

\smallskip
\textbf{Acknowledgments.} The author would like to express his
gratitude to A. Kucharski for several discussions.


\end{document}